\documentclass[3p, 12pt]{elsarticle}

\usepackage{amsmath,amssymb,amsthm}
\usepackage{natbib}

\newcommand{\EE}{\mathbb{E}}
\newcommand{\NN}{\mathbb{N}}
\newcommand{\PP}{\mathbb{P}}
\newcommand{\RR}{\mathbb{R}}
\newcommand{\ZZ}{\mathbb{Z}}

\newcommand{\rd}{{\rm d}}

\newcommand{\bA}{\mathbf{A}}
\newcommand{\bV}{\mathbf{V}}
\newcommand{\bW}{\mathbf{W}}
\newcommand{\bw}{\mathbf{w}}
\newcommand{\bY}{\mathbf{Y}}

\newtheorem{theorem}{Theorem}

\newtheorem{proposition}[theorem]{Proposition}
\newtheorem{corollary}[theorem]{Corollary}
\newdefinition{remark}{Remark}

\journal{}

\makeatletter
\def\ps@pprintTitle{%
 \let\@oddhead\@empty
 \let\@evenhead\@empty
 \def\@oddfoot{\centerline{\thepage}}%
 \let\@evenfoot\@oddfoot}
\makeatother

\begin{document}

\begin{frontmatter}

\title{Equivalent Representations of Max-Stable Processes via $\ell^p$ Norms}
\author[mo]{Marco Oesting} \ead{oesting@mathematik.uni-siegen.de}
\address{Universit\"at Siegen, Department Mathematik, D-57068 Siegen, Germany}

\begin{abstract}
While max-stable processes are typically written as pointwise maxima over an infinite number 
of stochastic processes, in this paper, we consider a family of representations based on 
$\ell^p$ norms. This family includes both the construction of the Reich--Shaby \citep{reich-shaby-12}
model and the classical spectral representation by \citet{dehaan84} as special cases. 
As the representation of a max-stable process is not unique, we present formulae to switch
between different equivalent representations. We further provide a necessary and sufficient 
condition for the existence of a $\ell^p$ norm based representation in terms of the stable 
tail dependence function of a max-stable process. Finally, we discuss several properties of
the represented processes such as ergodicity or mixing.
\end{abstract}


\end{frontmatter}

\section{Introduction}

Arising as limits of rescaled maxima of stochastic processes, max-stable
processes play an important role in spatial and spatio-temporal extremes. 
A stochastic process $X=\{X(s),\,s \in S\}$ on a countable index set $S$ is called 
max-stable if there exist sequences $\{a_n(\cdot)\}_{n \in \NN}$ and 
$\{b_n(\cdot)\}_{n \in \NN}$ of functions $a_n: S \to (0,\infty]$ and 
$b_n: S \to \RR$ such that, for all $n \in \NN$,
$$\mathcal{L}(X) = \mathcal{L}\left(\max_{i=1}^n \frac {X_i - b_n}{a_n}\right),$$
where $X_i$, $i \in \NN$, are independent copies of $X$ and the maximum is
taken pointwise. From univariate extreme value theory, it is well-known that
the marginal distributions of $X$, if non-degenerate, are necessarily 
Generalized Extreme Value (GEV) distributions, i.e.\
$$\PP(X(s) \leq x) = \exp\left(-\left(1 + \xi(s) \frac{x-\mu(s)}{\sigma(s)}\right)^{-1/\xi(s)}\right), \quad 1 + \xi(s) \frac{x-\mu(s)}{\sigma(s)} > 0,$$
with $\xi(s) \in \RR$, $\mu(s) \in \RR$ and $\sigma(s)>0$. As max-stability is
preserved by marginal transformations, it is common practice in extreme value
theory to consider only one type of marginal distributions, e.g.\ the case 
that the shape parameter $\xi$ is positive. In this case, the marginal 
distributions are of $\alpha$-Fr\'echet type, i.e., up to affine 
transformations, the marginal distribution functions are of the form 
$$\Phi_\alpha(x) = \exp\left(-x^\alpha\right), \quad x>0,$$
for some $\alpha >0$. Here, we will focus on the case of max-stable processes
with unit Fr\'echet margins, i.e.\ $X(s) \sim \Phi_1$ for all $s \in S$. In this
case, $X$ is called a simple max-stable process. 
\medskip

By \citet{dehaan84}, the class of simple max-stable processes on $S$ 
can be fully characterized: A stochastic process $\{X(s), \, s \in S\}$ is 
simple max-stable if and only if it possesses the spectral representation 
\begin{equation} \label{eq:spec-repr}
  X(s) = \max_{i \in \NN} A_i V_i(s), \quad s \in S,
\end{equation}
where $\sum_{i\in\NN} \delta_{A_i}$ is a Poisson point process on $(0,\infty)$ 
with intensity measure $a^{-2}\rd a$ and $V_i=\{V_i(s),\,s\in S\}$ are 
independent copies of a stochastic process $V$ such that $\EE(V(s))=1$ for all
$s \in S$ \citep[see also][]{GHV90,penrose92}. It is important to note that 
this representation is not unique. As different representations of the same
max-stable process might be convenient for different purposes such as 
estimation \citep[see][among others]{EMOS14, EMKS15} or simulation 
\citep[cf.][for instance]{OKS12,DM15,OSZ13}, finding novel representations is 
of interest. 
\medskip

Recently, \citet{reich-shaby-12} came up with a class of max-stable
processes written as a product
\begin{equation} \label{eq:reich-shaby}
 X(s) = U^{(p)}(s) \cdot \left[\sum_{l=1}^L B_l w_l(s)^p\right]^{1/p}, \qquad s \in S,
\end{equation}
where $\{U^{(p)}(s)\}_{s \in S}$ is a noise process with $U^{(p)}(s) \sim_{iid} \Phi_p$,
the functions $w_l: S \to [0,\infty)$, $l=1,\ldots,L$, are deterministic weight
functions such that $\sum_{l=1}^L w_l(s)=1$ for all $s \in S$ and, 
independently from $\{U^{(p)}(s)\}_{s \in S}$, the independent random variables
$B_l$, $l=1,\ldots,L$, follow a stable law given by the Laplace transform
$$ \EE\{\exp(-t \cdot B_l)\} = \exp(-t^{-1/p}), \qquad t>0.$$
The parameter $p \in (1,\infty)$ determines the strength of the effect of the
noise process which -- analogously to the terminology in geostatistics -- is 
also called a nugget effect. In \citet{reich-shaby-12}, the weight functions
$w_l$ are chosen as shifted and appropriately rescaled Gaussian density 
functions yielding an approximation of the well-known Gaussian extreme value
process \citep{smith90} joined with a nugget effect. Similarly, 
\citet{reich-shaby-12} propose analogues to popular max-stable processes such as
extremal Gaussian processes \citep{schlather02} and Brown-Resnick processes 
\citep{KSH09} by choosing appropriately rescaled realizations of Gaussian and 
log-Gaussian processes, respectively, as weight functions. Due to the flexibility
in modeling the strength of the nugget by the additional parameter $p$ and the 
tractability of the likelihood which allows to embed the model in a hierarchical 
Bayesian model, the Reich--Shaby model \eqref{eq:reich-shaby} has found its way 
into several applications \citep[cf.][for instance]{shaby-reich-12,RSC14,SSRS15,SFM17}.
\medskip

While a simple max-stable process in the spectral representation
\eqref{eq:spec-repr} is written as the pointwise supremum of an infinite number 
of processes, i.e.\ the pointwise $\ell_\infty$ norm of the random sequence 
$\{A_i \cdot W_i(s)\}_{i\in\NN}$, the Reich--Shaby model \eqref{eq:reich-shaby} 
is represented as the pointwise $p$ norm of the finite random vector 
$(B_l^{1/p} \cdot w_l(s))_{l=1,\ldots,L}$. 
In this paper, we will present a more general class of representations of
max-stable processes by writing them as pointwise $\ell^p$ norms of sequences
of stochastic processes, including both de Haan's representation and the 
Reich--Shaby model as special cases. The finite-dimensional distributions
of the resulting processes will turn out to be generalized logistic mixtures
introduced by \citet{FNR09} and \citet{FMN13}.
\medskip

This paper is structured as follows: In Section \ref{sec:representation},
we will introduce the spectral representation based on $\ell^p$ norms.
As a single max-stable process might allow for equivalent $\ell^p$ norm
based representations for different $p \in (1,\infty]$, we give formulae
to switch between them in Section \ref{sec:transformation}.
Section \ref{sec:characterization} provides a full characterization
of the resulting class of processes whose properties are finally discussed
in Section \ref{sec:dependence}.

\section{Generalization of the Spectral Representation} \label{sec:representation}

Denoting by 
\begin{equation*} 
 \|\bA \circ \bV(s)\|_p = \begin{cases} 
                           \left[ \sum_{i \in \NN} (A_i \cdot V_i(s))^p\right]^{1/p}, & p \in (1,\infty),\\ 
                                  \max_{i \in \NN} A_i \cdot V_i(s), & p=\infty,
                          \end{cases}
\end{equation*}
the $\ell^p$ norm of the Hadamard product of the sequences 
$\bA = \{A_i\}_{i \in \NN}$ and $\bV(s) = \{V_i(s)\}_{i \in \NN}$, $s \in S$, 
the spectral representation \eqref{eq:spec-repr} can be rewritten as
\begin{equation*}
  X(s) = \|\bA \circ \bV(s)\|_\infty, \quad s \in S.
\end{equation*}
We present a more general representation replacing the $\ell^\infty$ norm by a general 
$\ell^p$ norm, $p \in (1,\infty]$, and multiplication by an independent noise 
process with $\Phi_p$ marginal distributions. Here, we use the convention that
$\Phi_\infty$ denotes the weak limit of $\Phi_p$ as $p \to \infty$, i.e.\
$\Phi_\infty(x) = \mathbf{1}_{[1,\infty)}(x)$ is a degenerate distribution 
function.

\begin{theorem} \label{thm:max-stability}
Let $p \in (1,\infty]$ and $\{U^{(p)}(s)\}_{s\in S}$ be a collection of 
independent $\Phi_p$ random variables. Further, let $\sum_{i \in \NN} 
\delta_{A_i}$ be a Poisson process on $(0,\infty)$ with intensity $a^{-2}\rd a$
and $W_i^{(p)}$, $i \in \NN$, be independent copies of a stochastic process 
$\{W^{(p)}(s), \, s \in S\}$ with $\EE\{W^{(p)}(s)\}=1$ for all $s \in S$.
Then, the process $X$, defined by
\begin{align} \label{eq:p-repr}
 X(s) = \frac{U^{(p)}(s)}{\Gamma(1-p^{-1})} \|\bA \circ \bW^{(p)}(s)\|_p, \qquad s \in S,
\end{align}
is simple max-stable.
\end{theorem}
\begin{proof}
For $p=\infty$, we have $U^{(p)}(s) =1$ a.s.\ and, thus, representation 
\eqref{eq:p-repr} is of the same form as representation \eqref{eq:spec-repr}. 
Consequently, max-stability follows from \citet{dehaan84}. 

For $p \in (1,\infty)$, we first show that $\|\bA \circ \bW(s)\|_p < \infty$ 
a.s. According to Campbell's Theorem \citep[cf.][p.28]{kingman93}, this holds 
true if and only if
\begin{equation} \label{eq:2bfin}
\EE\left( \int_0^\infty \min\{|aW^{(p)}(s)|^p,1\} a^{-2} \, \rd a\right)<\infty.
\end{equation}
Substituting $v = a W(s)$, we can easily see that the left-hand side of
\eqref{eq:2bfin} equals
\begin{equation*}
  \EE\left(W^{(p)}(s)\right) \cdot \int_0^\infty \min\{|v|^p,1\} v^{-2} \, \rd v = 1 + \frac 1 {p-1}.
\end{equation*}
Thus, $\|\bA \circ \bW^{(p)}(s)\|_p < \infty$ a.s.
Then, for $s_1,\ldots,s_n \in S$, $x_1,\ldots,x_n>0$, $n \in \NN$, we obtain
\begin{align*}
   & \PP(X(s_i) \leq x_i, \, i=1,\ldots,n) \\
={}& \EE\left(\PP\left( U(s_i) \leq \frac{\Gamma(1-p^{-1}) x_i}{\|\bA \circ \bW^{(p)}(s_i)\|_p}, \, i=1,\ldots,n \, \Big| \, \bA, \bW^{(p)}\right)\right)\\
={}& \EE\left(\exp\left(- \sum_{i=1}^n \left( \frac{\Gamma(1-p^{-1}) x_i}{\|\bA \circ \bW^{(p)}(s_i)\|_p}\right)^{-p}\right)\right).
\end{align*}
Using well-known results on the Laplace functional of Poisson point processes,
this yields
\begin{align}
   & \PP(X(s_i) \leq x_i, \, i=1,\ldots,n) \nonumber \\
={}& \exp\left(\EE\left(\int_0^{\infty}\left\{
       \exp\left(-\sum\nolimits_{i=1}^n \left(\frac{a W^{(p)}(s_i)}{\Gamma(1-p^{-1}) x_i}\right)^p \right)-1\right\}
                 a^{-2}\,\rd a\right)\right) \nonumber \\
={}& \exp\left(\EE \left( \left\| \left(\frac{W^{(p)}(s_i)}{x_i}\right)_{i=1}^n\right\|_p \right) \cdot \frac{1}{p\Gamma(1-p^{-1})} \cdot
           \int_0^\infty \left(e^{-a}-1\right)  a^{-1-p^{-1}} \, \rd a \right) \nonumber \\ 
={}& \exp\left(-\EE\left(\left\| \left(\frac{W^{(p)}(s_i)}{x_i}\right)_{i=1}^n\right\|_p\right)\right) \label{eq:fidi}
\end{align}
where we used Formula 3.478.2 in \citet{gradshteyn-ryzik-88}. Thus, for $m$
independent copies $X_1, \ldots, X_m$ of $X$, $m \in \NN$, the homogeneity of
the $\ell^p$ norm yields
\begin{equation*}
 \PP\left(\frac 1 m \max_{j=1}^m X_j(s_i) \leq x_i, \, i=1,\ldots,n\right) = \PP\left(X(s_i) \leq x_i, \, i=1,\ldots,n\right),
\end{equation*}
i.e.\ $Z$ is simple max-stable.
\end{proof}
\medskip

\begin{remark}
  Theorem \ref{thm:max-stability} could alternatively be verified by observing 
  that the process $T(s) = \|\bA \circ \bW^{(p)}(s)\|_p^p$, $s \in S$, is 
  $\alpha$-stable with $\alpha=1/p$ (see also the proof of Theorem
  \ref{thm:character}). Thus, all the  finite-dimensional distributions of $X$ 
  are generalized logistic mixtures \citep[cf.][]{FNR09,FMN13} and, consequently, 
  are max-stable distributions.
\end{remark}
\medskip

Noting that the finite-dimensional distributions of the Reich--Shaby model 
\eqref{eq:reich-shaby} are given by
\begin{equation*}
      \PP\left(X(s_i) \leq x_i, \, i=1,\ldots,n\right) 
  ={} \exp\left(-\sum\nolimits_{j=1}^L \left\|\left(\frac{w_j(s_i)}{x_i}\right)_{i=1}^n\right\|_p\right),
\end{equation*}
it can be easily seen that \eqref{eq:reich-shaby} is a special case of 
representation \eqref{eq:p-repr} where $W$ follows the discrete distribution
$\PP(W = L w_i) = 1/L$, $i=1,\ldots,L$. Further, the classical spectral
representation \eqref{eq:spec-repr} by \citet{dehaan84} can be recovered
from representation \eqref{eq:p-repr} with $p = \infty$.
\medskip

Analogously to the law of the spectral processes $\{V_i(s),\, s \in S\}_{i \in \NN}$
in representation \eqref{eq:spec-repr},  the law of the processes 
$\{W^{(p)}_i(s),\, s \in S\}_{i \in \NN}$ in the $\ell^p$ norm based representation
of a given process $\{X(s),\, i \in S\}$ is  not unique: Let $Y_i$, $i \in \NN$, be
independently and identically distributed random variables with $\EE(Y_i)=1$ which 
are independent from $\sum_{i \in \NN} \delta_{A_i}$ and $\{W^{(p)}(s), \, s \in S\}$.
Then, the processes $\{U^{(p)}(s) / \Gamma(1-p^{-1}) \cdot \| \bA \circ \bW^{(p)}(s)\|_p, \, s \in S\}$
and $\{U^{(p)}(s) / \Gamma(1-p^{-1}) \cdot \| \bA \circ \bY \circ \bW^{(p)}(s)\|_p, \, s \in S\}$
are equal in distribution.

Consequently, even for some fixed $p \in (1,\infty]$ representation \eqref{eq:p-repr}
of a simple max-stable process $X$ is not unique. Furthermore, there might be 
representations of type \eqref{eq:p-repr} with different $p$ for the same process
$X$. Such equivalent representations are discussed in the following section.

\section{Equivalent Representations} \label{sec:transformation}

By \citet{dehaan84}, the class of simple max-stable processes is fully covered
by the class of processes which allow for the spectral representation 
\eqref{eq:spec-repr}, i.e.\ representation \eqref{eq:p-repr} with $p=\infty$.
Thus, any $\ell^p$ norm based representation \eqref{eq:p-repr} with $p<\infty$ 
of a simple max-stable process can be transformed to an equivalent 
representation of type \eqref{eq:spec-repr}. This transformation is presented
in the following proposition. Even more generally, it is shown how a 
$\ell^q$ norm based representation can be derived from a $\ell^p$ norm based
representation with $p < q < \infty$. 

\begin{proposition} \label{prop:trafo-repr}
 Let $X$ be a simple max-stable process with representation \eqref{eq:p-repr}
 for some $p \in (1,\infty)$. Then, the following holds:
 \begin{enumerate}
  \item The process $X$ allows for the spectral representation 
      \eqref{eq:spec-repr} with
      \begin{equation} \label{eq:vw}
        V(\cdot) =_d \frac{U^{(p)}(\cdot)}{\Gamma(1-p^{-1})}  W^{(p)}(\cdot).
      \end{equation}
  \item For $q \in (p,\infty)$, the process $X$ satisfies
      \begin{equation} \label{eq:q-repr}
       X(\cdot) =_d \frac{U^{(q)}(\cdot)}{\Gamma(1-q^{-1})}  \|\bA \circ \bW^{(q)}(\cdot)\|_q,       
      \end{equation}
      where $\{U^{(q)}(s)\}_{s \in S}$ is a collection of independent $\Phi_q$
      random variables and $W^{(q)}_i$, $i \in \NN$, are independent copies of 
      a stochastic process $\{W^{(q)}(s), \, s \in S\}$ given by
      \begin{equation*}
       W^{(q)}(s) = \frac{\Gamma(1-q^{-1})}{\Gamma(1-p^{-1})} (T_{(p/q)}(s))^{p/q} \cdot W^{(p)}(s), \quad s \in S.
      \end{equation*}
      Here, independently from the process $W^{(p)}$, the collection 
      $\{T_{(p/q)}(s)\}_{s \in S}$ consists of independent stable random 
      variables whose law is given by the Laplace transform
      $$\EE\left(e^{-tT_{(p/q)}(s)}\right) = e^{-t^{p/q}}, \quad t \geq 0.$$
 \end{enumerate}
\end{proposition}
\begin{proof}
 \begin{enumerate}
  \item By comparing the finite-dimensional distributions of the processes 
    defined via \eqref{eq:spec-repr} and \eqref{eq:p-repr}, it suffices to show
    that
    \begin{equation} \label{eq:zz-vw}
      \frac 1 {\Gamma(1-p^{-1})} \EE\left(\left\|\left(\frac{U^{(p)}(s_i) W^{(p)}(s_i)}{x_i}\right)_{i=1}^n\right\|_\infty\right)
      = \EE\left(\left\|\left(\frac{W^{(p)}(s_i)}{x_i}\right)_{i=1}^n\right\|_p\right), 
    \end{equation}
    for all $s_1, \ldots, s_n \in S$, $x_1,\ldots,x_n>0$, $n \in \NN$. To this 
    end, we  first note that
    \begin{align*}
        \PP\left(\left\|\left(\frac{U^{(p)}(s_i) W^{(p)}(s_i)}{x_i}\right)_{i=1}^n\right\|_\infty \leq y \, \bigg| \, \bW^{(p)} \right)
     =  \exp\left(-\frac 1 {y^p} \sum_{i=1}^n \left(\frac{W(s_i)}{x_i}\right)^{p}\right), \quad y>0,
    \end{align*}
    that is, conditionally on $\bW^{(p)}$, the norm 
    $\|(U^{(p)}(s_i) W^{(p)}(s_i)/x_i)_{i=1}^n\|_\infty$
    follows a $p$-Fr\'echet distribution with scale parameter $\|(W^{(p)}(s_i)/x_i)_{i=1}^n\|_p$. Thus,
    \begin{align*}
          \EE\left(\left\|\left(\frac{U^{(p)}(s_i) W^{(p)}(s_i)}{x_i}\right)_{i=1}^n\right\|_\infty\right)
     ={}& \EE_W \left\{\EE \left(\left\|\left(\frac{U^{(p)}(s_i) W^{(p)}(s_i)}{x_i}\right)_{i=1}^n\right\|_\infty \, \bigg| \, \bW^{(p)} \right) \right\}\\
     ={}& \EE_W \left\{\Gamma(1-p^{-1})\left\|\left(\frac{W^{(p)}(s_i)}{x_i}\right)_{i=1}^n\right\|_p\right\},
    \end{align*}
    i.e.\ Equation \eqref{eq:zz-vw}.
  \item From the first part of the proposition, it follows that the right-hand 
    side of \eqref{eq:q-repr} allows for a spectral representation 
    \eqref{eq:spec-repr} where the spectral functions are independent copies of
    the process $\widetilde V$ given by
    $$ \widetilde V(\cdot) = \frac {U^{(q)}(\cdot) \cdot (T_{p/q}(\cdot))^{1/q}}{\Gamma(1-p^{-1})} \cdot W^{(p)}(\cdot),$$
    while the spectral functions of the process $X$ on the left-hand side of
    \eqref{eq:q-repr} are independent copies of the process $V$ given in 
    \eqref{eq:vw}. Conditioning on the value of the stable random variable 
    $T_{(p/q)}(s)$, it can be shown that the product $U^{(q)}(s) \cdot T_{(p/q)}(s)$
    has the distribution function $\Phi_p$ for all $s \in S$ \citep[cf.][]{FNR09}
    and, thus,  $\widetilde V(\cdot) =_d V(\cdot)$. 
 \end{enumerate}
\end{proof}
\medskip

\begin{remark}
  Even though the transformation in the second part of the proposition 
  requires $p < q < \infty$, the two cases $p=q$ and $q=\infty$ can be regarded
  as limiting cases.
 
  As $q \searrow p$, we obtain that $U^{(q)}(\cdot) \to_d U^{(p)}(\cdot)$ and
  $\{T_{(p/q)}(s)\}_{s \in S}$ converges in distribution to a collection of
  random variables which equal $1$ a.s. Thus, in the limit $p=q$, there is
  no transformation.
 
  As $q \to\infty$, we have that $\Gamma(1-q^{-1}) \to 1$ and each
  $U^{(q)}(s)$, $s \in S$, converges to $1$ a.s. Further, by Thm.~1.4.5 in 
  \citet{samo-taqqu-94}, for each $s \in S$, the random variable $T_{(p/q)}(s)$
  can be represented as 
  $\frac{1}{\Gamma(1-p/q)} \sum_{i \in \NN} (\tilde A_i Y_i)^{q/p}$ where
  $\{\tilde A_i\}_{i \in \NN}$ are the points of a Poisson point process on
  $(0,\infty)$ with intensity $\tilde a^{-2} \rd \tilde a$ and $Y_i$, 
  $i \in \NN$, are independently and identically distributed non-negative 
  random variables with expectation $1$. Thus, as $q \to \infty$,
  $$ \left(T_{(p/q)}(s)\right)^{1/q} 
    =_d \left(\frac{1}{\Gamma(1-p/q)} \sum\nolimits_{i \in \NN} (\tilde A_i Y_i)^{q/p}\right)^{1/q}
    \longrightarrow_d \max_{i \in \NN} (\tilde A_i Y_i)^{1/p} $$
  which has the distribution function $\Phi_p$. Consequently, 
  $\left(T_{(p/q)}(\cdot)\right)^{1/q}\to_d U^{(p)}(\cdot)$.
\end{remark}
\medskip

Denoting by $\mathcal{MS}$ the class of all simple max-stable processes and by
$\mathcal{MS}_p$ the class of simple max-stable processes allowing for a
$\ell^p$ norm based spectral representation \eqref{eq:p-repr}, Proposition
\ref{prop:trafo-repr} yields 
$$\mathcal{MS}_p \subset \mathcal{MS}_q \subset \mathcal{MS}_\infty = \mathcal{MS}, \quad 1 < p < q <\infty.$$
A full characterization of the class $\mathcal{MS}_p $ is given in the following
section.

\section{Existence of $\ell^p$ Norm Based Representations} \label{sec:characterization}

In the following, we will present a necessary and sufficient criterion for the 
existence of a $\ell^p$ norm based representation of a simple max-stable 
process $X$ in terms of the stable tail dependence functions of its 
finite-dimensional distributions. For a simple max-stable distribution 
$(X(s_1),\ldots,X(s_n))^\top$, its stable tail dependence function
$l_{s_1,\ldots,s_n}$ is defined via
\begin{align*}
 l_{s_1,\ldots,s_n}:{}& [0,\infty)^n \to [0,\infty)\\
                      & (x_1,\ldots,x_n) \mapsto -\log\left\{\PP\left(X(s_1) \leq \frac 1 {x_1}, \ldots, X(s_n) \leq \frac 1 {x_n}\right) \right\}. 
\end{align*}
From the spectral representation \eqref{eq:spec-repr}, we obtain the form
\begin{align} \label{eq:stdf-spec}
 l_{s_1,\ldots,s_n}(x) = \EE\left( \max_{i=1,\ldots,n} x_i W(s_i) \right), \quad x \in [0,\infty)^n. 
\end{align}
The stable tail dependence function is homogeneous and convex 
\citep[cf.][among others]{BGST04}. Further, from Equation \eqref{eq:stdf-spec}
together with dominated convergence, we can deduce that the stable tail
dependence function is continuous.

\begin{theorem} \label{thm:character}
 Let $\{X(s), \ s \in S\}$ a simple max-stable process and $p \in (1,\infty)$. 
 Then, the following statements are equivalent:
 \begin{itemize}
  \item[(i)]  $X$ possesses a $\ell^p$ norm based representation \eqref{eq:p-repr}.
  \item[(ii)] For all pairwise distinct $s_1,\ldots,s_n \in S$ and $n \in \NN$, the
        function $f_{s_1,\ldots,s_n}$, defined by
        $$ f^{(p)}_{s_1,\ldots,s_n}(x) = l_{s_1,\ldots,s_n}(x_1^{1/p},\ldots,x_n^{1/p}), \quad x=(x_1,\ldots,x_n) \in [0,\infty)^n,$$
        is conditionally negative definite on the additive semigroup 
        $[0,\infty)^n$, i.e.\ for all $x^{(1)}$, $\ldots$, $x^{(m)} \in [0,\infty)^n$ and $a_1,\ldots,a_m \in \RR$
        such that $\sum_{i=1}^m a_i=0$, we have
        \begin{equation} \label{eq:neg-def}
          \sum_{i=1}^m \sum_{j=1}^m a_i a_j f^{(p)}_{s_1,\ldots,s_n}(x^{(i)} + x^{(j)}) \leq 0.
        \end{equation}
 \end{itemize}
\end{theorem}
\begin{proof}
 Firstly, we show that (i) implies (ii). To this end, let $X$ be a simple 
 max-stable process with representation \eqref{eq:p-repr}. Then, from 
 \eqref{eq:fidi}, we obtain that
 \begin{align*}
  f^{(p)}_{s_1,\ldots,s_n}(x) ={}& -\log\left\{\PP\left(X(s_1) \leq \frac 1 {x_1^{1/p}}, \ldots, X(s_n) \leq \frac 1 {x_n^{1/p}}\right) \right\}\\
                        ={}& \EE\left\{ \left( \sum\nolimits_{i=1}^n x_i W^{(p)}(s_i)^p \right)^{1/p} \right\}, \qquad x=(x_1,\ldots,x_n) \in [0,\infty)^n.
 \end{align*}
 Now, let $w(s_1),\ldots, w(s_n) \geq 0$ be fixed. Then, by a straightforward 
 computation, it can be seen that the function 
 $x \mapsto \sum_{k=1}^n x_k w(s_k)^p$ is conditionally negative definite on 
 $[0,\infty)^n$. As the function $y \mapsto y^{1/p}$ is a Bernstein function 
 and the composition of a conditionally negative function and a Bernstein 
 function yields a conditionally negative definite function 
 \citep[][Thm. 3.2.9]{berg-etal-84}, the function 
 $x \mapsto \left( \sum_{k=1}^n x_k w(s_k)^p\right)^{1/p}$ is conditionally 
 negative definite, as well. Being a mixture, the same is true for 
 $f^{(p)}_{s_1,\ldots,s_n}$. 
 
 Secondly, we show that (ii) implies (i). From the conditionally negative 
 definiteness of $f^{(p)}_{s_1,\ldots,s_n}$ , it follows that $e^{-f^{(p)}_{s_1,\ldots,s_n}}$ 
 is positive definite on $[0,\infty)^n$ \citep[][Thm.\ 3.2.2]{berg-etal-84}.
 As $l_{s_1,\ldots,s_n}$ is non-negative and continuous, 
 $e^{-f^{(p)}_{s_1,\ldots,s_n}}$ is further bounded by $1$ and continuous. Thus, by 
 Thm.\ 4.4.7 in \citet{berg-etal-84}, there exists a unique finite measure 
 $\mu_{s_1,\ldots,s_n}$ on $[0,\infty)^n$ with Laplace transform
 \begin{align} \label{eq:laplace-mu}
 \mathcal{L}\mu_{s_1,\ldots,s_n}(x) = \int_{[0,\infty)^n} \exp\left(-\langle x, a\rangle\right) \mu(\rd a) = \exp\left(-f_{s_1,\ldots,s_n}(x)\right), \quad x \in [0,\infty)^n.
 \end{align}
 Because of $\mu_{s_1,\ldots,s_n}([0,\infty)^n) = \exp(-l_{s_1,\ldots,s_n}(0,\ldots,0)) = 1$,
 $\mu_{s_1,\ldots,s_n}$ is a probability measure. Further, 
 \begin{equation} \label{eq:stdf-proj}
  l_{s_1,\ldots,s_n}(x_1,\ldots,x_{i-1}, 0, x_{i+1}, \ldots,x_n) = l_{s_1,\ldots,s_{i-1},s_{i+1},\ldots,s_n}(x_1,\ldots,x_{i-1}, x_{i+1}, \ldots,x_n)
 \end{equation}
 for all $x = (x_1,\ldots,x_n) \in [0,\infty)^n$ and $i \in \{1,\ldots,n\}$ 
 implies that
 \begin{align*} 
  & \mu_{s_1,\ldots,s_n}(A_1 \times \ldots \times A_{i-1} \times [0,\infty) \times A_{i+1} \times \ldots \times A_n)\\
    ={}& \mu_{s_1,\ldots,s_{i-1},s_{i+1},\ldots,s_n}(A_1 \times \ldots \times A_{i-1} \times A_{i+1} \times \ldots \times A_n)
 \end{align*}   
 for all Borel sets $A_1,\ldots,A_n \subset [0,\infty)$ and $i \in \{1,\ldots,n\}$, 
 that is, the family $\{\mu_{s_1,\ldots,s_n}: \ s_1,\ldots,s_n \in S, \ n \in \NN\}$
 of probability measures satisfies the consistency conditions from Kolmogorov's
 existence theorem. Thus, there exists a stochastic process $\{T(s),\, s \in S\}$
 with finite-dimensional distributions $\mu_{\cdot}$.
 
 Now, let $\{U^{(p)}(s)\}_{s \in S}$ be a collection of independent $\Phi_p$ 
 random variables and 
 $$ \tilde X(s) = U^{(p)}(s) T(s)^{1/p}, \quad s \in S.$$
 Then, for all pairwise distinct $s_1,\ldots,s_n \in S$ and $x_1,\ldots,x_n >0$,
 we have
 \begin{align*}
     & \PP(\tilde X(s_1) \leq x_1, \ldots, \tilde X(s_n) \leq x_n) \\
  ={}& \EE\left\{ \PP\left(U^{(p)}(s_1) \leq \frac{x_1}{T^{1/p}(s_1)}, \ldots, U^{(p)}(s_n) \leq \frac{x_1}{T^{1/p}(s_n)} \, \bigg| \, T(s_1),\ldots,T(s_n)\right)\right\}\\
  ={}& \EE\left\{ \exp\left(- \sum_{i=1}^n \frac{T(s_i)}{x_i^p}\right) \right\},
 \end{align*}
 By Equation \eqref{eq:laplace-mu}, we obtain
 \begin{align*}
     \PP(\tilde X(s_1) \leq x_1, \ldots, \tilde X(s_n) \leq x_n) 
  ={}& \exp\left(-f^{(p)}_{s_1,\ldots,s_n}(x_1^{-p},\ldots,x_n^{-p})\right) \\
  ={}& \PP\left(X(s_1) \leq x_1, \ldots, X(s_n) \leq x_n\right).
 \end{align*}
 Thus, $X$ allows for the spectral representation
 \begin{align} \label{eq:repr-step1} 
   X(s) = U(s) T^{1/p}(s), \qquad s \in S.
 \end{align}
 Now, let $T^{(1)},\ldots, T^{(m)}$ be $m$ independent copies of $T$ for 
 $m \in \NN$. Then, for all $s_1,\ldots,s_n \in S$ and
 $x=(x_1,\ldots,x_n) \in [0,\infty)^n$, we have
 \begin{align*}
  \EE\left\{\exp\left(-\Bigl\langle x, \left(\sum\nolimits_{k=1}^m T^{(k)}(s_i)\right)_{i=1}^n \Bigr\rangle\right)\right\} 
  ={}& \left[\EE\left\{ \exp(-\langle x, (T(s_i))_{i=1}^n\rangle)\right\}\right]^m\\
  ={}& \exp(-m \cdot l_{s_1,\ldots,s_m}(x_1^{1/p},\ldots,x_n^{1/p})) \\
  ={}& \exp(-l_{s_1,\ldots,s_m}((m^p x_1)^{1/p},\ldots,(m^p x_n)^{1/p}))\\
  ={}& \EE\left\{ \exp(\langle x, m^p (T(s_i))_{i=1}^n\rangle)\right\},
 \end{align*}
 where we used the homogeneity of the stable tail dependence function. Hence, 
 for all $s_1,\ldots,s_n \in S$, the vectors $\left(\sum\nolimits_{k=1}^m T^{(k)}(s_i)\right)_{i=1}^n$
 and $m^p (T(s_i))_{i=1}^n$ have the same distribution, i.e.\ $\{T(s), \, s \in S\}$
 is an $\alpha$-stable process with $\alpha=1/p$. Thus, from Thm.\ 13.1.2 and
 Thm.\ 3.10.1 in \citet{samo-taqqu-94}, we can deduce that $\{T(s), \, s \in S\}$
 allows for the representation 
 \begin{align} \label{eq:repr-step2} 
 T(s) = \frac{1}{\Gamma(1-p^{-1})^p} \sum_{i \in \NN} A_i^{p} \tilde W_i(s), \quad s \in S,
 \end{align}
 where $\{A_i\}_{i \in \NN}$ are the points of a Poisson point process on
 $[0,\infty)$ with intensity $a^{-2}\rd a$ and $\{\tilde W_i(s), \, s \in S\}$
 are independent and identically distributed stochastic processes which are
 independent from $\{A_i\}_{i \in \NN}$ and satisfy $\EE(\tilde W_i(s)^{1/p})
 = l_s(1) = 1$ for all $s \in S$. Defining $W_i^{(p)}(s) = \tilde W_i(s)^{1/p}$,
 $s \in S$, $i \in \NN$, and plugging Equation \eqref{eq:repr-step2} into
 Equation \eqref{eq:repr-step1}, we obtain Equation \eqref{eq:p-repr}.
\end{proof}
\medskip

\begin{remark}
 Note that Theorem \ref{thm:character} assumes that, for each $s_1,\ldots,s_n \in S$,
 $\ell_{s_1,\ldots,s_n}$ is the stable tail dependence function of the simple 
 max-stable vector $(X(s_1),\ldots,X(s_n))^\top$. The conditional negative 
 definiteness of the function $f^{(p)}_{s_1,\ldots,s_n}$ is an additional 
 condition. In particular, it is  always satisfied for $p=\infty$ -- i.e.\ any 
 simple max-stable process allows for de Haan's \citep{dehaan84} spectral
 representation \eqref{eq:spec-repr} -- as $f^{(\infty)}_{s_1,\ldots,s_n} =
 l_{s_1,\ldots,s_n}(1,\ldots,1)$ is always conditionally negative definite. 
 
 In order to check whether a function $l_{s_1,\ldots,s_n}$ is the stable tail
 dependence function of some process $X$ with an $\ell^p$ norm based representation,
 we first need to ensure that $l_{s_1,\ldots,s_n}$ is a valid stable tail
 dependence function. This can be done by checking necessary and sufficient conditions
 given in \citet{molchanov08} and \citet{ressel13}, for instance.
\end{remark}
\medskip

Using an integral representation of continuous conditionally negative definite
functions on $[0,\infty)^n$ \citep[cf.][Paragraph 4.4.6]{berg-etal-84}, condition
(ii) in Theorem \ref{thm:character} can be reformulated yielding the following
corollary.

\begin{corollary}
 For a simple max-stable process $\{X(s), \ s \in S\}$ and $p \in (1,\infty)$, 
 the following statements are equivalent:
 \begin{itemize}
  \item[(i)]  $X$ possesses a $\ell^p$ norm based representation \eqref{eq:p-repr}.
  \item[(ii)] For all pairwise distinct $s_1,\ldots,s_n \in S$ and $n \in \NN$, there
        exist a vector $c(s_1,\ldots,s_n) = (c_1(s_1,\ldots,s_n),\ldots,c_n(s_1,\ldots,s_n))^\top \in [0,\infty)^n$         
        and a Radon measure $\mu_{s_1,\ldots,s_n}$ on $[0,\infty)^n$ such that
        the stable tail dependence function $l_{s_1,\ldots,s_n}$ satisfies
        $$ l_{s_1,\ldots,s_n}(x) = \sum_{i=1}^n c_i(s_1,\ldots,s_n) \cdot x_i^p 
                                              + \int_{[0,\infty)^n} \left\{1 - \exp\left(-\sum_{i=1}^n a_i x_i^p\right)\right\} \mu_{s_1,\ldots, s_n}(\mathrm{d}a),$$
        for all $x=(x_1,\ldots,x_n)^\top \in [0,\infty)^n$.  
 \end{itemize}
\end{corollary}
\medskip

From the characterization given in Theorem \ref{thm:character}, we can deduce
necessary conditions on the dependence structure of a max-stable process with 
$\ell^p$ norm based representation \eqref{eq:p-repr} in terms of its extremal
coefficients:  
For a general simple max-stable process $\{X(s), \, s \in S\}$ and a finite set
$\tilde S=\{s_1,\ldots,s_n\} \subset S$, let the extremal coefficient 
$\theta(\tilde S)$ be defined via
$$ \PP\left( \max_{s \in S} X(s) \leq x\right) = \exp\left(-\frac{\theta(\tilde S)} x\right), \quad x >0$$
Then, we necessarily have $\theta(\tilde S) \in [1,n]$ where 
$\theta(\tilde S)=n$ if and only if $X(s_1), \ldots, X(s_n)$ are independent 
and $\theta(\tilde S)=1$ if and only if $X(s_1)=X(s_2)=\ldots=X(s_n)$ a.s. The
extremal coefficient is closely connected to the stable tail dependence 
function via the relation
$$ \theta(\{s_1,\ldots,s_n\}) = l_{s_1,\ldots,s_n}(1,\ldots,1).$$ 
If $X$ further allows for an $\ell^p$ norm based representation 
\eqref{eq:p-repr}, we obtain the following condition.

\begin{proposition} \label{prop:gen-bound-ecf}
 Let $\{X(s), \, s \in S\}$ be a simple max-stable process with representation
 \eqref{eq:p-repr} and $S_1, S_2 \subset S$ be finite and disjoint. Then, we 
 have
 $$ \theta(S_1 \cup S_2) \geq 2^{1/p}\frac{\theta(S_1) + \theta(S_2)}{2}.$$
\end{proposition}
\begin{proof}
 Let $S_1 = \{s_1, s_2, \ldots, s_{k_1}\}$ and $S_2 = \{s_{k_1+1}, \ldots, 
 s_{k_1+k_2}\}$ and let $\{e_1,\ldots,e_{k_1+k_2}\}$ denote the standard basis
 in $\RR^{k_1+k_2}$. As the function $(x_1,\ldots,x_{k_1+k_2}) \mapsto 
 l_{s_1,\ldots,s_{k_1+k_2}}(x_1^{1/p},\ldots,x_{k_1+k_2}^{1/p})$ is conditionally 
 negative definite by Theorem \ref{thm:character}, inequality \eqref{eq:neg-def}
 particularly holds true for $n=2$, $a_1=1$, $a_2=-1$, 
 $x^{(1)} = \sum_{i=1}^{k_1} e_i$ and $x^{(2)} = \sum_{i=k_1+1}^{k_1+k_2} e_i$,
 i.e.\
 $$  l_{s_1,\ldots,s_{k_1+k_2}}\left(2^{1/p}\sum\nolimits_{i=1}^{k_1} e_i\right) 
  +  l_{s_1,\ldots,s_{k_1+k_2}}\left(2^{1/p}\sum\nolimits_{i=k_1+1}^{k_1+k_2} e_i\right) 
  -2 l_{s_1,\ldots,s_{k_1+k_2}}\left(\sum\nolimits_{i=1}^{k_1+k_2} e_i\right) \leq 0.$$
 Using the homogeneity and property \eqref{eq:stdf-proj} of the stable tail
 dependence function, we obtain
  $$ 2^{1/p} l_{s_1,\ldots,s_{k_1}}(1,\ldots,1) + 2^{1/p} l_{s_{k_1+1},\ldots,s_{k_1+k_2}}(1,\ldots,1) 
        - 2 l_{s_1,\ldots,s_{k_1+k_2}}(1,\ldots,1) \leq 0.$$
 As $\theta(\tilde S) = l_{\tilde S}(1,\ldots,1)$ for any finite $\tilde S 
 \subset S$, this yields the assertion.       
\end{proof}
\medskip

Of particular interest in extreme value analysis is the case of the pairwise
extremal coefficient function \citep[cf.][]{smith90,schlather-tawn-03} where
$\tilde S = \{s_1,s_2\}$. Then, Proposition \ref{prop:gen-bound-ecf} provides 
the lower bound
\begin{equation} \label{eq:easy-bound-ec}
 \theta(\{s_1,s_2\}) \geq 2^{1/p} \qquad \text{for all } s_1 \neq s_2 \in S.
\end{equation}
For the particular case of model \eqref{eq:reich-shaby}, this bound has already
been found by \citet{reich-shaby-12} motivating their interpretation of model 
\eqref{eq:reich-shaby} as a max-stable process with nugget effect in analogy to
the Gaussian case.
\medskip

The bound \eqref{eq:easy-bound-ec} and the characterization of simple max-stable processes with a $\ell^p$ 
norm based representation given in Theorem \ref{thm:character} can be used to show
the existence of a \emph{minimal} $\ell^p$ norm based representation of a simple
max-stable process $X$, i.e.\ the existence of some $p_{\min}(X)$ such that 
$X \in \mathcal{MS}_p$ if and only if $p \geq p_{\min}(X)$.

\begin{corollary} \label{coro:min-repr}
 Let $\{X(s), \, s \in S\}$ be a simple max-stable process such that not all
 random variables $\{X(s)\}_{s \in S}$ are independent. Then, there exists
 a number $p_{\min}(X) \in (1,\infty]$ such that $X \in \mathcal{MS}_p$ if and
 only if $p \geq p_{\min}(X)$.
\end{corollary}
\begin{proof}
 By \citet{dehaan84}, any simple max-stable process $X$ satisfies 
 $X \in \mathcal{MS}_\infty$. Thus, the assertion follows directly if
 $$ p_{\min}(X) = \inf \{p > 1: \, X \in \mathcal{MS}_p\} = \infty.$$
 Thus, we restrict ourselves to the case that $p_{\min}(X) < \infty$. As not 
 all the random variables $\{X(s)\}_{s \in S}$ are independent, there exist 
 $s_1, s_2 \in S$ and $\varepsilon > 0$ such that 
 $\theta(\{s_1,s_2\}) < 2^{1/(1+\varepsilon)}$. Hence, by Equation
 \eqref{eq:easy-bound-ec}, we obtain that $p_{\min}(X) \geq 1 + \varepsilon$.
 Using the fact that $\mathcal{MS}_p \subset \mathcal{MS}_q$ for $p<q$, it 
 remains to show that $X \in \mathcal{MS}_{p_{\min}(X)}$. By Theorem 
 \ref{thm:character}, for all pairwise distinct $s_1,\ldots, s_n \in S$, 
 $n \in \NN$, $a_1,\ldots,a_m \in \RR$ such that $\sum_{i=1}^m a_i=0$, 
 $x^{(1)},\ldots,x^{(m)} \in [0,\infty)^n$ and $m \in \NN$ we have that
 $$ \sum_{i=1}^m \sum_{j=1}^n a_i a_j l_{s_1,\ldots,s_n}((x_1^{(i)}+x_1^{(j)})^{1/p},\ldots,(x_n^{(i)}+x_n^{(j)})^{1/p}) \leq 0$$
 for all $p > p_{\min}(X)$. By the continuity of $l_{s_1,\ldots,s_m}$, the same 
 holds true for $p=p_{\min}(X)$, and, thus, by Theorem \ref{thm:character},
 $X \in \mathcal{MS}_{p_{\min}(X)}$.
\end{proof}
\medskip

For any $p \in (1,\infty]$, we now give an example for a simple max-stable 
process $X^{(p)}$ such that $p_{\min}(X^{(p)}) = p$. Thus, we will also see
that
$$\mathcal{MS}_p \subsetneq \mathcal{MS}_q \subsetneq \mathcal{MS}_\infty = \mathcal{MS}, \quad 1 < p < q <\infty.$$
We consider the process $X_{\log}^{(p)} \in \mathcal{MS}_p$ which possesses
an $\ell^p$ norm based representation \eqref{eq:p-repr} with $W(s)=1$ a.s.\ for
all $s \in S$. From Equation \eqref{eq:fidi}, for pairwise distinct $s_1,\ldots,s_n \in S$,
we obtain the finite-dimensional distributions
\begin{equation} \label{eq:logistic}
 \PP\left(X_{\log}^{(p)}(s_i) \leq x_i, \ 1 \leq i \leq n\right) = \exp\left\{-\left(\sum\nolimits_{i=1}^n x_i^{-p}\right)^{1/p}\right\}, \qquad x_1,\ldots,x_n>0,
\end{equation}
i.e.\ all the multivariate distributions are multivariate logistic 
distributions \citep{gumbel60}. Thus, the process $X_{\log}^{(p)}$ has pairwise
extremal coefficients $\theta(s,t) = 2^{1/p}$ for all $s,t \in S$, $s \neq t$. 
From Equation \eqref{eq:easy-bound-ec}, it follows that $X_{\log}^{(p)} \notin \mathcal{MS}_{p'}$
for $p'<p$. Consequently, we have $p_{\min}(X_{\log}^{(p)}) = p$.

\section{Properties of Processes with $\ell^p$ Norm Based Representation} \label{sec:dependence}

In this section, we will analyze several properties of simple max-stable 
processes with an $\ell^p$ norm based representation in more detail. We will 
particularly focus on properties related to the dependence structure of the
process such as stationarity, ergodicity and mixing. A characteristic feature
of a process $X$ with $\ell^p$ norm based representation \eqref{eq:p-repr} is 
the additional noise introduced via the process $\{U^{(p)}(s), \ s \in S\}$.
Thus, we will compare the process $X$ to a ``denoised'' reference process
$$ \overline{X}(s) = \max_{i \in \NN} A_i W_i^{(p)}(s), \quad s \in S,$$
i.e.\ the simple max-stable process constructed via the same spectral functions
used in the original ($\ell^\infty$ norm based) spectral representation \eqref{eq:spec-repr}.

\begin{proposition}  \label{prop:ec-bounds}
 Let $\{X(s), \ s \in S\}$ be a simple max-stable process with $\ell^p$ norm 
 based representation \eqref{eq:p-repr} with $p \in (1\infty]$. Then, for the 
 pairwise extremal coefficients $\theta(\{s_1,s_2\})$, we obtain the bounds:
 $$ \EE\left(\max\{W^{(p)}(s_1),W^{(p)}(s_2)\}\right) \leq \theta(\{s_1,s_2\}) \leq 2^{1/p} \left[\EE\left(\max\{W^{(p)}(s_1),W^{(p)}(s_2)\}\right)\right]^{1-p^{-1}}.$$
\end{proposition}
\begin{proof}
 In the case $p=\infty$, we have
 $$ \theta(\{s_1,s_2\}) = \EE\left(\max\{W^{(p)}(s_1),W^{(p)}(s_2)\}\right),$$
 which equals both the lower and the upper bound given in the assertion.
 
 Now, let $p \in (1,\infty)$. Then, we have the lower bound
 \begin{align*}
  \theta(\{s_1,s_2\}) {}={} \EE\left\{\left(W^{(p)}(s_1)^p+W^{(p)}(s_2)^p\right)^{1/p}\right\} 
                   {}\geq{} \EE\left(\max\{W^{(p)}(s_1),W^{(p)}(s_2)\}\right).
 \end{align*}
 Further, for any $p < r < \infty$ and $\bw \in [0,\infty)^2$, we obtain
 $$ \|\bw\|^p_p \leq \|\bw\|_1^{\frac{r-p}{r-1}} \cdot \|\bw\|_r^{r\frac{p-1}{r-1}}$$
 \citep[cf.][Thm.~18]{HLP52}, or equivalently
 $$ \|\bw\|_p \leq \|\bw\|_1^{\frac{1}{p}\frac{r-p}{r-1}} \cdot \|\bw\|_r^{\frac{1-p^{-1}}{1-r^{-1}}}.$$
 As $r \to \infty$, this yields
 $$ \|\bw\|_p \leq \|\bw\|_1^{1/p} \cdot \|\bw\|_\infty^{1-p^{-1}}.$$
 Taking the expectation of $\bw$ with respect to the joint distribution of
 $W^{(p)}(s_1)$ and $W^{(p)}(s_2)$ and applying H\"older's inequality, we obtain the upper
 bound
 \begin{align*}
  \theta(\{s_1,s_2\}) ={}& \EE\left\{\left(W^{(p)}(s_1)^p+W^{(p)}(s_2)^p\right)^{1/p}\right\} \\
                   \leq{}& \EE\left\{\left(W^{(p)}(s_1)+W^{(p)}(s_2)\right)^{1/p} \cdot \max\{W^{(p)}(s_1),W^{(p)}(s_2)\}^{1-p^{-1}}\right\}\\
                   \leq{}& \left[\EE\left\{W^{(p)}(s_1)+W^{(p)}(s_2)\right\}\right]^{1/p}  \left[\EE\left(\max\{W^{(p)}(s_1),W^{(p)}(s_2)\}\right)\right]^{1-p^{-1}}. 
 \end{align*} 
 The assertion follows from $\EE\{W^{(p)}(s_1)\} = \EE\{W^{(p)}(s_2)\} = 1$. 
\end{proof}
\medskip

Note that Proposition \ref{prop:ec-bounds} relates the extremal coefficients
$\theta(\{s_1,s_2\})$, $s_1,s_2 \in S$, to the terms
$\EE\left(\max\{W^{(p)}(s_1),W^{(p)}(s_2)\}\right)$ which are the extremal coefficients of 
the process
$$ \overline{X}(s) = \max_{i \in \NN} A_i W^{(p)}(s), \qquad s \in S.$$
As the processes $X$ and $\overline X$ just differ by the Fr\'echet noise
process $U^{(p)}$, we will call $\overline X$ the denoised max-stable
process associated to $X$. From Proposition \ref{prop:ec-bounds}, we obtain 
that extremal dependence of the process $X$ is always weaker than dependence of
the associated denoised process -- as expected.
\medskip

In the following, we will consider the case that $S=\ZZ$. In this case, 
properties such as stationarity, ergodicity or mixing are of interest. For a 
simple max-stable $\{X(s),\,s \in \ZZ\}$ with representation 
\eqref{eq:spec-repr}, necessary and sufficient conditions for these properties
can be expressed in terms of the distribution of the spectral function $V$:
By \citet{KSH09}, $X$ is stationary if and only if
\begin{equation} \label{eq:stationarity}
 \EE\left\{ V(s_1)^{u_1} \cdot \ldots \cdot V(s_n)^{u_n} \right\} = \EE \left\{ V(s_1+s)^{u_1} \cdot \ldots \cdot V(s_n+s)^{u_n} \right\}
\end{equation}
for all $n \in \NN$, $s, s_1,\ldots,s_n \in \ZZ$ and $u_1,\ldots,U_n \in [0,1]$ 
such that $\sum_{i=1}^n u_i = 1$. 
For stationary simple max-stable processes, \citet{kabluchko-schlather-10} give
conditions for ergodicity and mixing in terms of the pairwise extremal 
coefficients $\theta(\{s_1,s_2\}) = \EE(\max\{V(s_1),V(s_2)\})$, stating that
$X$ is mixing if and only 
\begin{equation} \label{eq:mixing}
 \lim_{r \to \infty} \theta(\{0,r\}) = 2,
\end{equation}
and $X$ is ergodic if and only if
\begin{equation} \label{eq:ergodicity}
 \lim_{r \to \infty} \frac 1 r \sum\nolimits_{k=1}^r \theta(\{0,k\}) = 2,
\end{equation}
respectively. 

Now, we transfer these results to a max-stable process $X$ with $\ell^p$ norm
based representation \eqref{eq:p-repr} giving necessary and sufficient 
conditions in terms of $W^{(p)}$. For the associated denoised process
$\overline{X}$, Equations \eqref{eq:stationarity}--\eqref{eq:ergodicity} depend
on the distribution $W^{(p)} = V$ only, while the structure of the process
$X$ is more difficult as we have $V(\cdot) = [\Gamma(1-p^{-1})]^{-1} U^{(p)}(\cdot) W^{(p)}(\cdot)$
(cf.\ Proposition \ref{prop:trafo-repr}). The following result, however, shows
that those conditions simplify to the conditions for the associated denoised 
process $\overline{X}$.

\begin{proposition} \label{prop:prop}
 Let $\{X(s), \, s \in \ZZ\}$ be a simple max-stable process with $\ell^p$ norm
 based representation \eqref{eq:p-repr} and let $\overline{X}$ be the denoised
 process associated to $X$. Then, the following holds:
 \begin{enumerate}
  \item[1.] $X$ is stationary if and only if $\overline{X}$ is stationary.
 \end{enumerate}
 If $X$ is stationary, we further have
 \begin{enumerate} 
  \item[2.] $X$ is mixing if and only if $\overline{X}$ is mixing.
  \item[3.] $X$ is ergodic if and only if $\overline{X}$ is ergodic.
 \end{enumerate}
\end{proposition}
\begin{proof}
 \begin{enumerate}
  \item[1.]  By \citet{KSH09} and Proposition \ref{prop:trafo-repr}, the
    process $X$ is stationary if and only if \eqref{eq:stationarity} holds for 
    $V(\cdot) = [\Gamma(1-p^{-1})]^{-1} U^{(p)}(\cdot) W^{(p)}(\cdot)$. The 
    left-hand side of \eqref{eq:stationarity} equals
    \begin{align*}
     \EE\left\{ V(s_1)^{u_1} \cdot \ldots \cdot V(s_n)^{u_n} \right\} 
    ={}& \frac 1 {\Gamma(1-p^{-1})} \EE\left\{\prod\nolimits_{i=1}^n U^{(p)}(s_i)^{u_i} W^{(p)}(s_i)^{u_i}\right\}\\
    ={}& \frac 1 {\Gamma(1-p^{-1})} \EE\left\{\prod\nolimits_{i=1}^n U^{(p)}(s_i)^{u_i}\right\} \EE\left\{ \prod\nolimits_{i=1}^n  W^{(p)}(s_i)^{u_i}\right\}\\
    ={}& \frac {\prod_{i=1}^n \Gamma(1-u_i p^{-1})} {\Gamma(1-p^{-1})} \EE\left\{\prod\nolimits_{i=1}^n  W^{(p)}(s_i)^{u_i}\right\},
    \end{align*}
    where we used the fact that $U^{(p)}(s_i)^{u_i}$, $i=1,\ldots,n$, are 
    independent $\Phi_{p/u_i}$ random variables. Thus, $X$ is stationary if and
    only if Equation \eqref{eq:stationarity} holds for $V=W^{(p)}$, i.e.\ if
    and only if $\overline{X}$ is stationary.
  \item[2.]  By \citet{kabluchko-schlather-10}, the process $X$ is mixing if 
    and only if Equation \eqref{eq:mixing} holds. where $\theta$
    denotes the pairwise extremal coefficient of $X$. Proposition 
    \ref{prop:ec-bounds} yields the bounds
    \begin{align*}
    \lim_{r \to \infty} \EE(\max\{W^{(p)}(0),W^{(p)}(r)\}) \leq{}& \lim_{r \to \infty} \theta(\{0,r\})\\
                  \leq{}& 2^{1/p} \lim_{r \to \infty} \left[\EE(\max\{W^{(p)}(0),W^{(p)}(r)\})\right]^{1-p^{-1}} \leq 2.
    \end{align*}
    Thus, $ \lim_{r \to \infty} \theta(\{0,r\}) =2$ if and only if
    $\lim_{r \to \infty} \EE(\max\{W^{(p)}(s_1),W^{(p)}(s_2)\})=2$ which is equivalent to
    the fact that $\overline{X}$ is mixing as $\EE(\max\{W^{(p)}(0),W^{(p)}(r)\})$ is the
    extremal coefficient of $\overline{X}$.
  \item[3.] The proof runs analogously to the proof of the second assertion.
    The process $X$ is ergodic if and only if Equation \eqref{eq:ergodicity}
    holds. From Proposition \ref{prop:ec-bounds} and Jensen's inequality, we
    obtain
    \begin{align*}
            & \phantom{2^{1/p}} \lim_{r \to \infty} r^{-1} \sum\nolimits_{k=1}^r \EE(\max\{W^{(p)}(0),W^{(p)}(k)\}) 
    {}\leq{} \lim_{r \to \infty} r^{-1} \sum\nolimits_{k=1}^r \theta(\{0,k\})\\
      \leq{}& 2^{1/p} \lim_{r \to \infty} r^{-1} \sum\nolimits_{k=1}^r \left[\EE(\max\{W^{(p)}(0),W^{(p)}(k)\})\right]^{1-p^{-1}}\\
      \leq{}& 2^{1/p} \lim_{r \to \infty} \left[r^{-1} \sum\nolimits_{k=1}^r \EE(\max\{W^{(p)}(0),W^{(p)}(k)\})\right]^{1-p^{-1}} \leq 2.
    \end{align*}
    Consequently, we have that $\lim_{r \to \infty} r^{-1} \sum_{k=1}^r \theta(\{0,k\}) = 2$
    holds true if and only if $\lim_{r \to \infty} r^{-1} \sum_{k=1}^r \EE(\max\{W^{(p)}(0),W^{(p)}(k)\})=2$
 \end{enumerate}
\end{proof}
\medskip

\begin{remark}
  The mixing properties of a stochastic process $\{X(s), \, s \in S\}$ are
  described more precisely by its mixing coefficients. For two subsets $S_1, 
  S_2 \subset S$, the $\beta$-mixing coefficient $\beta(S_1,S_2)$ is defined by
  $$ \beta(S_1,S_2) = \sup\{| \mathcal{P}_{S_1 \cup S_2}(C) - \mathcal{P}_{S_1} \otimes \mathcal{P}_{S_2}(C)|, \, C \in \mathcal{C}_{S_1 \cup S_2}\},$$ 
  where, for each $\tilde S \subset S$, the probability measure 
  $\mathcal{P}_{\tilde S}$ denotes the distribution of the restricted process
  $\{X(s), \, s \in \tilde S\}$ on the space of non-negative functions on 
  $\tilde S$ endowed with the Borel-$\sigma$ algebra $\mathcal{C}_{\tilde S}$.
  \medskip
  
  For the case of a max-stable process, \citet{DEM12} provide the upper bound
  $$\beta(S_1,S_2) \leq 4 \sum_{s_1 \in S_1} \sum_{s_2 \in S_2} [2 - \theta(s_1,s_2)].$$
  Applying Proposition \ref{prop:ec-bounds} , we obtain  
  $$\beta(S_1,S_2) \leq 4 \sum_{s_1 \in S_1} \sum_{s_2 \in S_2} [2 - \theta(s_1,s_2)] \leq 4 \sum_{s_1 \in S_1} \sum_{s_2 \in S_2} [2 - \EE(\max\{W^{(p)}(s_1),W^{(p)}(s_2)\})],$$
  i.e.\ the upper bound for a process with $\ell^p$ norm based representation
  \eqref{eq:p-repr} is lower than the bound for the associated denoised process.
\end{remark}
\medskip

As Proposition \ref{prop:prop} states, a max-stable process with $\ell^p$ norm 
based representation \eqref{eq:p-repr} shares properties such as stationary,
ergodicity and mixing with the associated denoised process. In particular, the 
``noisy'' analogues of well-studied max-stable processes might be used without
changing any of these properties.

\section*{Acknowledgements}

The author is grateful to Prof.\ Stilian Stoev and Dr.\ Kirstin Strokorb
for pointing out some connections to other work.

\section*{References}

\bibliographystyle{elsarticle-harv}
\bibliography{lit}

\end{document}